\documentclass[12pt]{amsart}
\usepackage{amsmath}
\usepackage{setspace}
\usepackage[top=3cm,bottom=3cm,left=3cm,right=3cm]{geometry}

\usepackage{amsthm}
\newtheorem{lemma}{Lemma}
\newtheorem{thmm}[lemma]{Theorem}
\newtheorem{corr}[lemma]{Corollary}
\theoremstyle{definition}
\newtheorem*{defi}{Definition}

\begin{document}
\doublespacing

\title[Note on the group-theoretic approach to fast matrix multiplication]{A note on the group-theoretic approach to\\fast matrix multiplication}

\author{Ivo Hedtke}
\address{Mathematical Institute, University of Jena, D-07737 Jena, Germany}
\thanks{The author was supported by the Studienstiftung des Deutschen
Volkes.}
\email{Ivo.Hedtke@uni-jena.de}

\maketitle

\begin{abstract}
In 2003 \textsc{Cohn} and \textsc{Umans} introduced a group-theoretic approach to fast matrix multiplication.
This involves finding large subsets $S$, $T$ and $U$ of a group $G$ satisfying the Triple Product Property (TPP) as a means to bound the exponent $\omega$ of the matrix multiplication. We show that $S$, $T$ and $U$ may be be assumed to contain the identity and be otherwise disjoint. We also give a much shorter proof of the upper bound $|S|+|T|+|U|\leq |G|+2$.
\end{abstract}

\section{Introduction}
\noindent The naive algorithm for matrix multiplication is an $\mathcal 
O(n^3)$ algorithm. From \textsc{Volker Strassen} (\cite{Strassen}) we know that there is an $\mathcal
O(n^{2.81})$ algorithm for this problem.
\textsc{Winograd} optimized
\textsc{Strassen}'s algorithm. While the \textsc{Strassen-Winograd} algorithm is
the variant that is always implemented (for example in the famous GEMMW package),
there are faster ones (in theory) that are impractical to implement. The fastest
known algorithm runs in $\mathcal O(n^{2.376})$ time (see \cite{Coppersmith} from \textsc{Don Coppersmith} and \textsc{Shmuel Winograd}). Most researchers believe that an optimal
algorithm with $\mathcal O(n^2)$ runtime exists, but since 1987 no further
progress was made in finding one.

Because modern architectures have complex memory hierarchies and increasing
parallelism, performance has become a complex tradeoff, not just a simple matter
of counting flops. Algorithms which make use of this technology were described in \cite{Alberto} by \textsc{D'Alberto} and  \textsc{Nicolau}. An
also well known method is \emph{Tiling}: The normal algorithm can be speeded up
by a factor of two by using a six loop implementation that blocks submatrices so
that the data passes through the L1 Cache only once.

In 2003 \textsc{Cohn} and \textsc{Umans} introduced in \cite{Cohn} a group-theoretic approach to fast matrix multiplication. The main idea is to embed the matrix multiplication over a ring $R$ into the group ring $RG$, where $G$ is a (finite) group. A group $G$ admits such an embedding, if there are subsets $S$, $T$ and $U$ which fulfill the so called \emph{Triple Product Property}. 

\begin{defi}[Right Quotient]
Let $G$ be a group and $\emptyset\neq X \subseteq G$ be a nonempty subset of $G$. The \emph{right quotient} $Q(X)$ of $X$ is defined by $Q(X):=\{xy^{-1} : x,y \in X\}$.
\end{defi}

\begin{defi}[TPP]
We say that the nonempty subsets $S$, $T$, and $U$ of a group $G$ fulfill the \emph{Triple Product Property} (TPP) if for $s\in Q(S)$, $t\in Q(T)$ and $u \in Q(U)$,
$stu=1$ holds iff $s=t=u=1$.
\end{defi}

\textsc{Cohn} and \textsc{Umans} found a way to bound the exponent $\omega$ of the matrix multiplication with their framework.
Therefore, for a fixed group $G$ we search for TPP triples $S$, $T$ and $U$ which maximize $|S|\cdot |T|\cdot |U|$, for example with a brute force computer search. Here one can use \textsc{Murthy}'s upper bound (s. Corollary \ref{KOR}) and our intersection condition (s. Theorem \ref{MainResult}).

\section{Results}

\noindent We show that $S$, $T$ and $U$ may be be assumed to contain the identity and be otherwise disjoint.

\begin{thmm}\label{MainResult}
If $S'$, $T'$ and $U'$ fulfill the TPP, then there exists a triple $S$, $T$ and $U$ with
\[
|S|=|S'|, \quad |T|=|T'|, \quad |U|=|U'| \qquad and \qquad S\cap T = T \cap U = S \cap U= 1
\]
which also fulfills the TPP.
\end{thmm}

For the proof of our main result we need some auxiliary results.

\begin{lemma}\label{lemm:BruteRechts}
Let $\emptyset \neq X\subseteq G$ be a nonempty subset of a group $G$ and $g\in G$. Then
\begin{enumerate}
\item $1\in Q(X)$,
\item $g \in Q(X) \Leftrightarrow g^{-1} \in Q(X)$ and
\item $|X| \leq |Q(X)|$.
\end{enumerate}
\end{lemma}

\begin{proof}
\begin{enumerate}
\item Because $X\neq \emptyset$ there exists an $x \in X$ and so $1=xx^{-1}\in Q(X)$ follows.
\item If $g\in Q(X)$ then there are $x,y\in X$ with $g=xy^{-1}$. This implies, that $g^{-1}=(xy^{-1})^{-1}=yx^{-1} \in Q(X)$.
\item For a fixed $x \in X$ the map $X \to Q(X)$, $y\mapsto yx^{-1}$ is injective and therefore $|X| \leq |Q(X)|$ holds.\qedhere
\end{enumerate}
\end{proof}

\begin{lemma}\label{lemm:SchnittQ}
If $S$, $T$ and $U$ fulfill the TPP then \[Q(X)\cap Q(Y)=1\] holds for all $X\neq Y \in \{S,T,U\}$.
\end{lemma}

\begin{proof}
We know $1\in Q(X)\cap Q(Y)$ from Lemma \ref{lemm:BruteRechts}(1). Now assume that $|Q(X)\cap Q(Y)| \geq 2$. In this case there is an $1\neq x \in Q(X) \cap Q(Y)$. From Lemma \ref{lemm:BruteRechts}(2) we know, that $x^{-1} \in Q(X) \cap Q(Y)$, too. Moreover 1 is an element of every right quotient and therefore the factors $x$, $x^{-1}$ and $1$ occur in $\{stu : s\in Q(S), t\in Q(T), u \in Q(U)\}$ and the TPP is not fulfilled. So we have $|Q(X)\cap Q(Y)| = 1$ which completes the proof.
\end{proof}

Theorem \ref{SatzSchnittZweier} and Corollary \ref{KOR} below are originally due to \textsc{Murthy} (2009). Our proofs are somewhat shorter.

\begin{thmm}[\textsc{Murthy}'s minimal disjointness property]\label{SatzSchnittZweier}
If $S$, $T$ and $U$ fulfill the TPP then \[|X \cap Y|\leq 1\] holds for all $X\neq Y \in \{S,T,U\}$.
\end{thmm}

\begin{proof}
Assume that $|X \cap Y|\geq 2$. Then there are $x\neq y \in X \cap Y$. Therefore we have $1\neq xy^{-1}\in Q(X) \cap Q(Y)$. This is a contradiction to Lemma \ref{lemm:SchnittQ}.
\end{proof}

Now we can prove our main result.
\begin{proof}[Proof of Theorem \ref{MainResult}]
We fix $s_0\in S'$, $t_0\in T'$ and $u_0\in U'$. Now we define $S:=\{ss_0^{-1} : s\in S'\}$ and $T$ and $U$ in the same way. Obviously $|S|=|S'|$, $|T|=|T'|$ and $|U|=|U'|$ holds. Because of
\[
Q(S)=\{s\tilde s^{-1} : s,\tilde s \in S\}=\{ss_0^{-1}(\tilde ss_0^{-1})^{-1}: s,\tilde s \in S'\}=\{s\tilde s^{-1} : s,\tilde s \in S'\}=Q(S'),\]
$Q(T)=Q(T')$ and $Q(U)=Q(U')$ the triple $S$, $T$ and $U$ fulfill the TPP, too. It is also clear, that $1\in S$, $1\in T$ and $1\in U$. The result now follows from Theorem~\ref{SatzSchnittZweier}.
\end{proof}

Finally we can prove the upper bound of $|G| + 2$ for the additive size of a TPP triple.

\begin{thmm}\label{satz:OrderQ}
If $S$, $T$ and $U$ fulfill the TPP then $|Q(S)|+|Q(T)|+|Q(U)| \leq |G| + 2$.
\end{thmm}

\begin{proof}
Note that $Q(S)\cup Q(T) \cup Q(U) \subset G$ and
\begin{align*}
&|Q(S)\cup Q(T) \cup Q(U)|\\
&= |Q(S)|+|Q(T)|+|Q(U)| - |Q(S)\cap Q(T)| - |Q(T)\cap Q(U)|- |Q(S)\cap Q(U)|\\
 &+ |Q(S)\cap Q(T) \cap Q(U)|.
\end{align*}
Because of Lemma \ref{lemm:SchnittQ} all intersections have size $1$ and the statement follows.
\end{proof}

\begin{corr}[\textsc{Murthy}]\label{KOR}
If $S$, $T$ and $U$ fulfill the TPP then $|S|+|T|+|U| \leq |G| + 2$.
\end{corr}

\begin{proof}
The statement follows from Lemma \ref{lemm:BruteRechts}(3) and Theorem \ref{satz:OrderQ}.
\end{proof}

Note that Theorem \ref{satz:OrderQ} is more effective than Corollary \ref{KOR} when searching for TPP triples.

\section*{Acknowledgements}
\noindent I would like to thank \textsc{David J. Green} for our inspiring discussions.

\end{document}